\documentclass[reqno, 12pt]{amsart}
\pdfoutput=1
\makeatletter
\let\origsection=\section \def\section{\@ifstar{\origsection*}{\mysection}}
\def\mysection{\@startsection{section}{1}\z@{.7\linespacing\@plus\linespacing}{.5\linespacing}{\normalfont\scshape\centering\S}}
\makeatother

\usepackage{amsmath,amssymb,amsthm}
\usepackage{mathrsfs}
\usepackage{mathabx}\changenotsign
\usepackage{dsfont}
\usepackage[dvipsnames]{xcolor}

\usepackage{xcolor}
\usepackage[backref]{hyperref}
\hypersetup{
    colorlinks,
    linkcolor={red!60!black},
    citecolor={green!60!black},
    urlcolor={blue!60!black}
}

\usepackage[dvipsnames]{xcolor}
\usepackage[open,openlevel=2,atend]{bookmark}

\usepackage[abbrev,msc-links,backrefs]{amsrefs}
\usepackage{doi}

\renewcommand{\PrintDOI}[1]{\doi{#1}}

\usepackage[T1]{fontenc}
\usepackage{lmodern}
\usepackage[babel]{microtype}
\usepackage[english]{babel}

\usepackage{geometry}
\numberwithin{equation}{section}
\numberwithin{figure}{section}

\usepackage{enumitem}
\def\rmlabel{\upshape({\itshape \roman*\,})}

\def\ag#1{	\tikz{\def\nn{#1};
	\pgfmathsetmacro\wie{3*\nn-1};
	\pgfmathsetmacro\mm{5*\nn};
	\pgfmathsetmacro\kk{\nn-1};
	
	\foreach \i in {0,...,\mm}{
		\coordinate (x\i) at (\i*360/\wie:1.3cm);}
		
	\foreach \i in {0,...,\wie}{	
		\foreach \j [evaluate=\j as \k using \i+\j+\nn] in {0,...,\kk}{
			\draw[green!75!black] (x\i)--(x\k);}
	}		
	
	\foreach \i in {0,...,\wie}{	
		\draw[black, fill=black] (x\i) circle (1pt);
	}		}}

\newcommand{\mref}[1]{\ifmmode\textrm{\ref{#1}}\else\ref{#1}\fi}

\let\polishlcross=\l
\def\l{\ifmmode\ell\else\polishlcross\fi}

\def\qand{\quad\text{and}\quad}

\let\emptyset=\varnothing
\let\setminus=\smallsetminus
\let\sm=\smallsetminus

\makeatletter
\def\moverlay{\mathpalette\mov@rlay}
\def\mov@rlay#1#2{\leavevmode\vtop{   \baselineskip\z@skip \lineskiplimit-\maxdimen
   \ialign{\hfil$\m@th#1##$\hfil\cr#2\crcr}}}
\newcommand{\charfusion}[3][\mathord]{
    #1{\ifx#1\mathop\vphantom{#2}\fi
        \mathpalette\mov@rlay{#2\cr#3}
      }
    \ifx#1\mathop\expandafter\displaylimits\fi}
\makeatother

\newcommand{\dcup}{\charfusion[\mathbin]{\cup}{\cdot}}

\DeclareFontFamily{U}  {MnSymbolC}{}
\DeclareSymbolFont{MnSyC}         {U}  {MnSymbolC}{m}{n}
\DeclareFontShape{U}{MnSymbolC}{m}{n}{
    <-6>  MnSymbolC5
   <6-7>  MnSymbolC6
   <7-8>  MnSymbolC7
   <8-9>  MnSymbolC8
   <9-10> MnSymbolC9
  <10-12> MnSymbolC10
  <12->   MnSymbolC12}{}
\DeclareMathSymbol{\powerset}{\mathord}{MnSyC}{180}

\usepackage{tikz}
\usetikzlibrary{calc,decorations.pathmorphing}
\usetikzlibrary{math}
\pgfdeclarelayer{background}
\pgfdeclarelayer{foreground}
\pgfdeclarelayer{front}
\pgfsetlayers{background,main,foreground,front}

\usepackage{subcaption}

\let\epsilon=\varepsilon
\let\eps=\epsilon
\let\rho=\varrho
\let\theta=\vartheta

\def\ZZ{{\mathds Z}}

\def\RR{{\mathds R}}

\newcommand{\Ex}{\mathscr{E}
	(n,s)}
\newcommand{\ex}{\mathrm{ex}(n,s)}

\newcommand{\Nn}{\mathrm{N}}

\newcommand{\ccM}{\mathscr{M}}

\newcommand{\al}{\alpha}

\newtheoremstyle{note}  {4pt}  {4pt}  {\sl}  {}  {\bfseries}  {.}  {.5em}          {}
\newtheoremstyle{introthms}  {3pt}  {3pt}  {\itshape}  {}  {\bfseries}  {.}  {.5em}          {\thmnote{#3}}
\newtheoremstyle{remark}  {2pt}  {2pt}  {\rm}  {}  {\bfseries}  {.}  {.3em}          {}

\theoremstyle{plain}
\newtheorem{thm}{Theorem}[section]
\newtheorem{lemma}[thm]{Lemma}
\newtheorem{prop}[thm]{Proposition}

\newtheorem{conj}[thm]{Conjecture}
\newtheorem{cor}[thm]{Corollary}
\newtheorem{fact}[thm]{Fact}
\newtheorem{claim}[thm]{Claim}

\theoremstyle{note}
\newtheorem{dfn}[thm]{Definition}

\theoremstyle{remark}

\usepackage{accents}

\usepackage{mathtools}
\usepackage{lineno}
\newcommand*\patchAmsMathEnvironmentForLineno[1]{%
\expandafter\let\csname old#1\expandafter\endcsname\csname #1\endcsname
\expandafter\let\csname oldend#1\expandafter\endcsname\csname end#1\endcsname
\renewenvironment{#1}%
{\linenomath\csname old#1\endcsname}%
{\csname oldend#1\endcsname\endlinenomath}}%
\newcommand*\patchBothAmsMathEnvironmentsForLineno[1]{%
\patchAmsMathEnvironmentForLineno{#1}%
\patchAmsMathEnvironmentForLineno{#1*}}%
\AtBeginDocument{%
\patchBothAmsMathEnvironmentsForLineno{equation}%
\patchBothAmsMathEnvironmentsForLineno{align}%
\patchBothAmsMathEnvironmentsForLineno{flalign}%
\patchBothAmsMathEnvironmentsForLineno{alignat}%
\patchBothAmsMathEnvironmentsForLineno{gather}%
\patchBothAmsMathEnvironmentsForLineno{multline}%
}
\usepackage{multicol}
\usepackage{todonotes}

\usepackage{marginnote}

\def\GZS{\mathbf{Sym}}
\newcommand{\G}{\Gamma}

\begin{document}


\title[On the Ramsey-Tur\'an density of triangles]{On the Ramsey-Tur\'an density of triangles}

\author[T.~\L uczak]{Tomasz \L uczak}
\address{Adam Mickiewicz University, Faculty of Mathematics and Computer Science, Pozna\'n, Poland}
\email{tomasz@amu.edu.pl}
\thanks{The first author is supported by National Science Centre, Poland, grant 2017/27/B/ST1/00873.}

\author[J.~Polcyn]{Joanna Polcyn}
\address{Adam Mickiewicz University, Faculty of Mathematics and Computer Science, Pozna\'n, Poland}
\email{joaska@amu.edu.pl}

\author[Chr.~Reiher]{Christian Reiher}
\address{Fachbereich Mathematik, Universit\"at Hamburg, Hamburg, Germany}
\email{Christian.Reiher@uni-hamburg.de}


\begin{abstract}
One of the oldest results in modern graph theory, due to Mantel, asserts that every 
triangle-free graph on $n$ vertices has at most $\lfloor n^2/4\rfloor$ edges. 
About half a century later Andr\'asfai studied dense triangle-free graphs and proved 
that the largest
triangle-free graphs on $n$ vertices without independent sets of size $\alpha n$, 
where $2/5\le \alpha < 1/2$, are blow-ups of the pentagon. More than 50 further years 
have elapsed since Andr\'asfai's work. In this article we make the next step  
towards understanding the structure of dense triangle-free graphs without large 
independent sets.

Notably, we determine the maximum size of triangle-free graphs~$G$ on $n$ vertices 
with $\alpha (G)\ge 3n/8$ and state a conjecture on the structure of the densest 
triangle-free graphs $G$ with $\alpha(G) > n/3$. We remark that the case $\alpha(G) \le n/3$
behaves differently, but due to the work of Brandt this situation is fairly well understood.  
\end{abstract}

\maketitle



\section{Introduction}
\subsection{Ramsey-Tur\'an theory}
Mantel~\cite{M} proved in 1907 that balanced complete bipartite graphs maximise the number 
of edges among all triangle-free graphs on a given set of vertices. Later 
this result was generalized to $K_\ell$-free graphs by Tur\'an~\cite{T} and, asymptotically,
to $H$-free graphs by Erd\H{o}s and Stone~\cite{ES}, and by Erd\H{o}s and 
Simonovits~\cite{ESim}. All of these works had a decisive impact on
the development of extremal graph theory. 

Here we deal with one particular line of research systematically initiated by Vera~T.~S\'os
and known as Ramsey-Tur\'an theory (for a survey on this fascinating area 
we refer to Simonovits and S\'os~\cite{SS}). 
For $\ell\ge 3$ and $n\ge s\ge 0$ the {\sl Ramsey-Tur\'an number 
$\mathrm{ex}_\ell(n,s)$} is the maximum number of edges in a $K_\ell$-free graph on $n$ 
vertices which contains no independent set consisting of more than $s$ vertices, i.e.,
\[
	\mathrm{ex}_\ell(n,s)=\max_{G=(V,E)}\bigl\{|E|\colon K_\ell\nsubseteq G,\;|V|=n, 
	\textrm{ and } \al(G)\le s\bigl\}\,.
\]

One usually considers the case that $n$ is large and thus, 
instead of $\mathrm{ex}_\ell(n,s)$, one rather studies the 
{\sl Ramsey-Tur\'an density function 
$f_\ell\colon (0,1]\longrightarrow\RR$} defined by
\begin{equation*}
	f_\ell(\alpha)
	=
	\lim\limits_{n\rightarrow\infty}\frac{\mathrm{ex}_\ell(n,\alpha n)}{\binom n2}\,.
\end{equation*}
For the existence of this limit we refer to~\cite{EHSS}. Notice that Tur\'an's theorem implies 
\begin{equation}\label{eq:1908}
	\mathrm{ex}_\ell(n,s)
	=
	\bigl(1-o(1)\bigr) \frac{\ell-2}{\ell-1} \binom n2 
	\quad  \text{\ for\ } 
	s\ge \Big\lceil \frac{n}{\ell-1}\Big\rceil\,,
\end{equation}
whence $f_\ell(\alpha)=\frac{\ell-2}{\ell-1}$ for all $\alpha\ge \frac 1{\ell-1}$. 
The central problem to determine the limit
\[
	\rho(K_\ell)=\lim_{\alpha\to 0} f_\ell(\alpha)
\]
was solved for odd $\ell$ in~\cite{ES69} and proved to be very difficult for even $\ell$. 
Following the ingenious contributions by Szemer\'edi~\cite{Sz72} and by Bollob\'as 
and Erd\H{o}s~\cite{BE76}, this problem was finally solved by Erd\H{o}s, Hajnal, S\'os, 
and Szemer\'edi~\cite{EHSS}. 

Recently L\"uders and Reiher~\cite{LR-a} determined the value of $f_\ell(\alpha)$ 
whenever $\alpha$ is sufficiently small depending on $\ell$, so the extremal behaviour 
of $K_\ell$-free graphs with small (linear) independence number is now well understood. 
Despite this fact, we firmly believe that the problem to determine~$f_\ell(\alpha)$ 
{\it for all} $\alpha\in \bigl(0, \frac 1{\ell-1}\bigr)$ remains interesting. 

\subsection{Results on triangle-free graphs}

In this article we restrict our attention to the innocent looking case $\ell=3$, 
which seems surprisingly intricate to us. In other words, we concentrate on 
triangle-free graphs and, eliminating some indices, we study the behaviour of 
the function $\ex=\mathrm{ex}_3(n,s)$ and its `scaled' version $f(\alpha)=f_3(\alpha)$.
We also write
\[
	\Ex
	=
	\bigl\{G=(V,E)\colon K_3\nsubseteq G,\;|V|=n, \;\al(G)\le s,\textrm{ and } |E|=\ex\bigr\}
\]
for the corresponding families of extremal graphs. 

In an early contribution from 1962 Andr\'asfai~\cite{A} studied the question of 
how many edges a triangle-free graph on $n$ vertices, whose independence number 
is at most~$\alpha n$ for some given constant $\alpha$, can have. 
As a special case of~\eqref{eq:1908}, for $\alpha \ge \frac 12$ Mantel's Theorem 
yields~$\bigl\lfloor \frac{n^2}{4}\bigr\rfloor$ as an answer and for $\alpha \le \frac 13$ 
the ``trivial'' upper bound $\frac 12 \alpha n^2$ is essentially optimal 
(see the discussion at the end of this subsection), so the question is most interesting when 
$\alpha\in \bigl(\frac 13,\frac 12\bigr)$.
Andr\'asfai settled this problem for all $\alpha\in \bigl[\frac 25,\frac 12\bigr]$ 
and conjectured for $\alpha\in \bigl(\frac 13,\frac 12\bigr)$ that the answer is describable 
in terms of appropriate blow-ups of certain graphs nowadays bearing his name (see 
Conjecture~\ref{c:1} below). 

\begin{thm}[Andr\'asfai] \label{thm:1730}
	For every nonnegative integer $n$ and every integer $s\in\bigl[\frac25 n, \frac12 n\bigr]$
	we have 
	\[
		 \ex=n^2 - 4ns + 5s^2\,.
	\]
	In particular, $f(\alpha)=2-8\alpha +10\alpha^2$ holds for 
	every $\alpha\in\bigl[\frac 25, \frac 12\bigr]$.
	
\end{thm}  

Andr\'asfai~\cite{A} also determined the extremal families $\Ex$ for 
$s\in\bigl[\frac25 n, \frac12 n\bigr]$ and it turned out that all extremal graphs 
in these families are blow-ups of the pentagon. We display a sample case in 
Figure~\ref{fig:pentagon} and defer a more detailed discussion to  
Subsection~\ref{subsec:1242}.  
Our main result is a similar quadratic formula applicable to every
$\alpha\in \bigl[\frac38, \frac25\bigr]$.

\begin{figure}[ht]
	\centering
	\begin{tikzpicture}[scale=0.8]
	
	\coordinate (x1) at (-2.5,4);
	\coordinate (x2) at (2.5,4);
	\coordinate (x3) at (-4.2,2);
	\coordinate (x4) at (4.2,2);
	\coordinate (x5) at (0,0);
	\coordinate (x6) at (-3,3.7);
	
	\draw [black!20!white, line width=16pt] (x1) -- (x2);
	\draw [black!20!white, line width=18pt] (x2) -- (x4);
	\draw [black!20!white, line width=18pt] (x1) -- (x3);
	\draw [black!20!white, line width=18pt] (x4) -- (x5);
	\draw [black!20!white, line width=18pt] (x5) -- (x3);
	
	\begin{pgfonlayer}{front}
	\foreach \i in {1,2}{
		\draw[blue!75!black, line width=2pt,fill opacity=1] (x\i) ellipse (1.8cm and 18pt);
		\fill[blue!15!white,opacity=1] (x\i) ellipse (1.8cm and 18pt);
		\node at (x\i) {\large\textcolor{blue!75!black}{$3s-n$}};
	}	
	\end{pgfonlayer}
	
	\begin{pgfonlayer}{front}
\foreach \i in {3,4,5}{
	\draw[red!75!black, line width=2pt,fill opacity=1] (x\i) ellipse (1.3cm and 15pt);
\fill[red!10!white,opacity=1] (x\i) ellipse (1.3cm and 15pt);
\node at (x\i) {\large\textcolor{red!75!black}{$n-2s$}};
}	
\end{pgfonlayer}

	\end{tikzpicture}
	\caption{Example of a graph in $\Ex$ for $2n/5\le s\le n/2$.}
	\label{fig:pentagon} 
\end{figure}

\begin{thm}\label{thm:2338} 
	If $n\ge 0$ and $s\in \bigl[\frac 38 n, \frac 25n\bigr]$, then	
	\[
		\ex = 3n^2 - 15ns + 20s^2\,.
	\]
	Consequently, we have 
	\[
		f(\alpha)= 6-30\alpha+40\alpha^2
		\quad\textrm{for every}\quad 
		\alpha\in\bigl[\tfrac 38, \tfrac 25\bigr]\,.  
	\]
\end{thm} 

It follows from our proof that all extremal graphs for this result, i.e., all 
graphs in a class of the form $\Ex$ with $s\in \bigl[\frac 38 n, \frac 25n\bigr]$
are blow-ups of the well-known Wagner graph here denoted by $\Gamma_3$, 
which is a  triangle-free cubic graph on $8$ vertices. 
A special case is shown in Figure~\ref{fig:A8}. 

\begin{figure}[ht]
	\centering
	\begin{tikzpicture}[scale=0.5]
	
	\coordinate (x1) at (-3,5);
	\coordinate (x2) at (3,5);
	\coordinate (x3) at (8,2);
	\coordinate (x4) at (8,-2);
	\coordinate (x5) at (3,-5);
	\coordinate (x6) at (-3,-5);
	\coordinate (x7) at (-8,-2);
	\coordinate (x8) at (-8,2);
	
	\draw [black!20!white, line width=8pt] ($(x1)+(-.6,0)$) -- ($(x4)+(0,-.3)$);
	\draw [black!20!white, line width=8pt] ($(x1)+(-.6,0)$)  -- ($(x5)+(.6,0)$) ;
	\draw [black!20!white, line width=8pt] ($(x1)+(-.6,0)$)  --  ($(x6)+(-.6,0)$);
	\draw [black!20!white, line width=8pt] ($(x2)+(.6,0)$) -- ($(x5)+(.6,0)$) ;
	\draw [black!20!white, line width=8pt] ($(x2)+(.6,0)$)  --  ($(x6)+(-.6,0)$);
	\draw [black!20!white, line width=8pt] ($(x2)+(.6,0)$)  -- ($(x7)+(0,-.3)$);
	\draw [black!20!white, line width=8pt] ($(x3)+(0,.2)$) -- ($(x6)+(-.6,0)$) ;
	\draw [black!20!white, line width=8pt] ($(x3)+(0,.2)$) -- ($(x7)+(0,-.3)$);
	\draw [black!20!white, line width=8pt] ($(x3)+(0,.2)$) -- ($(x8)+(0,.2)$);
	\draw [black!20!white, line width=8pt] ($(x4)+(0,-.3)$) -- ($(x7)+(0,-.3)$);
	\draw [black!20!white, line width=8pt] ($(x4)+(0,-.3)$) -- ($(x8)+(0,.2)$);
	\draw [black!20!white, line width=8pt] ($(x5)+(.6,0)$)  -- ($(x8)+(0,.2)$);
	
	\begin{pgfonlayer}{front}
	\foreach \i in {3,5,8}{
		\draw[red!75!black, line width=2pt,fill opacity=1] (x\i) ellipse (1.6cm and 18pt);
		\fill[red!15!white,opacity=1] (x\i) ellipse (1.6cm and 18pt);
		\node at (x\i) {\large\textcolor{red!75!black}{{\small $2n-5s$}}};
	}	
	\end{pgfonlayer}
	
	\begin{pgfonlayer}{front}
	\foreach \i in {1,2,4,6,7}{
			\draw[blue!75!black, line width=2pt,fill opacity=1] (x\i) ellipse (2cm and 21pt);
			\fill[blue!25!white,opacity=1] (x\i) ellipse (2cm and 21pt);
			\node at (x\i) {\large\textcolor{blue!75!black}{$3s-n$}};
	}	
	\end{pgfonlayer}
	
	\end{tikzpicture}
	\caption{Example of a graph in $\Ex$ for $3n/8\le s\le 2n/5$.}
	\label{fig:A8}
\end{figure}

We suspect that some of the tools we have developed for proving Theorem~\ref{thm:2338}
will be relevant for a complete determination of the function $f$, even though some 
new ideas will certainly be required. Before stating our version of Andr\'asfai's 
conjecture on $f$ we briefly recall another known result on this function.

Notice that every triangle-free graph $G$ satisfies
\begin{equation*}
	\Delta(G)\le \alpha (G)\,,
\end{equation*}
for the neighbourhood of every vertex is an independent set. 
Therefore
\[
	\ex\le \tfrac12 ns
\]
holds for all $n\ge s\ge 0$ and $f(\alpha)\le \alpha$ for $\alpha>0$ follows. 
We call these estimates the {\sl trivial bounds} on $\ex$ and $f(\alpha)$, respectively.  
  
In the regime $s < \frac13 n$ Brandt~\cite{B10} provided several constructions of $s$-regular 
graphs on~$n$ vertices whose independence number is equal to $s$, and his work implies 
$f(\alpha) = \alpha$ for all $\alpha \in \bigl(0, \frac 13\bigr]$. 
In view of this result 
and the above theorems it remains to study the behaviour of $f(\alpha)$ 
for $\alpha\in \bigl(\frac13, \frac 38\bigr)$. The next subsection offers a conjecture 
for this range.

\subsection{A conjecture on triangle-free graphs}\label{subsec:1242}
Let us introduce one more piece of terminology. By a {\sl blow-up} of a graph $G$ we mean 
any graph $\hat G$ obtained from $G$ be replacing each of its vertices $v_i$ by an 
independent set $V_i$ (that can be empty) and joining two subsets~$V_i$ and~$V_j$ of 
vertices of $\hat G$ by all $|V_i||V_j|$ possible edges whenever the pair $\{v_i,v_j\}$ 
is an edge of~$G$. For instance, Figure~\ref{fig:pentagon} shows a blow-up of the pentagon,
where three consecutive vertices are replaced by independent $(n-2s)$-sets (drawn red) 
while the remaining two vertices are enlarged to (blue) independent $(3s-n)$-sets.

Let us recall that Andr\'asfai graphs are Cayley graphs $\bigl(\ZZ/(3k-1)\ZZ, S\bigr)$
with $k\ge 1$ and $S\subseteq \ZZ/(3k-1)\ZZ$ being a sum-free subset of size~$k$.
For definiteness, we denote the graph obtained for $S=\{k, \ldots, 2k-1\}$
by $\Gamma_k$. So explicitly two vertices $i$ and $j$ of $\Gamma_k$ are adjacent if and only 
if $i-j\in S$. Following $\Gamma_1=K_2$ the first few Andr\'asfai graphs are depicted 
in Figure~\ref{fig:ag}. 

\begin{figure}[ht]
		\centering
		\begin{multicols}{5}
			\ag{2} \\ \ag{3} \\ \ag{4} \\ \ag{5} \\ \ag{6}
		\end{multicols}
		\caption{Andr\'asfai graphs $\G_2$, $\G_3$, $\G_4$, $\G_5$, and $\G_6$.}
		\label{fig:ag}
\end{figure}

Notice that $\Gamma_k$ is a triangle-free, $k$-regular graph on $3k-1$ vertices 
whose independence number is exactly $k$. Therefore, balanced blow-ups of $\Gamma_k$ 
show that the trivial bound on~$f(\alpha)$ is optimal if $\alpha$ is of the form
$\frac{k}{3k-1}$, i.e., that we have 
\begin{equation}\label{eq:1803}
	f\Big(\frac k{3k-1}\Big)
	=
	\frac k{3k-1}\quad \text{\ for all\ }\ k\ge 1\,.
\end{equation}

Like Andr\'asfai we believe that whenever $n\ge 0$ and $s\in\bigl(\frac13 n, \frac12 n\bigr]$
there exists a graph $G\in \Ex$ which is a blow-up of an appropriate Andr\'asfai graph. 
This leads to $f(\alpha)$ being piecewise quadratic on $\bigl(\frac13, \frac 12\bigr]$
with critical values at $\alpha_k=\frac{k}{3k-1}$ for $k\ge 2$. By optimizing over all
blow-ups of Andr\'asfai graphs we were led to the following function. 

\begin{dfn}\label{dfn:1707}
	For integers $n\ge s\ge 0$ we set 
	\[
		g(n, s)=
		\begin{cases}
			\frac 12ns\,, & \text{ if $s\le \frac13 n$} \cr
					g_k(n, s)\,, & 
					\text{ if $\frac{k}{3k-1}n \le s < \frac{k-1}{3k-4}n$ for some $k\ge 2$} \cr
			\bigl\lfloor\frac{n^2}{4}\bigr\rfloor\,, & \text{ if $\frac n2 \le s\le n$,}  
		\end{cases}
	\]
	where 
	\begin{equation}\label{eq:1727}
		g_k(n, s)=\tfrac12 k(k-1)n^2-k(3k-4)ns+\tfrac12(3k-4)(3k-1)s^2\,.
	\end{equation}
\end{dfn} 

\begin{conj}\label{c:1}
	For all integers $n\ge s\ge 0$ we have $\ex\le g(n, s)$. In other words, 
	every triangle-free $n$-vertex graph $G$ with $\alpha(G)\le s$ has at most $g(n, s)$
	edges. 
\end{conj}
	
Admittedly, one needs some time to get used to the functions $g_k(n, s)$ introduced
in~\eqref{eq:1727} but we believe that the motivation in terms of optimal blow-ups of 
Andr\'asfai graphs renders the conjecture sufficiently natural (see also~\cite{Vega}*{Lemma~3.3}). 
Observe that
the functions~$g_2(n, s)$ and $g_3(n, s)$ are precisely the quadratic forms appearing 
in the Theorems~\ref{thm:1730} and~\ref{thm:2338}. Therefore, Conjecture~\ref{c:1}
is only open for $s\in \bigl(\frac13 n, \frac 38 n\bigr)$.    
Let us briefly indicate one construction showing that, if true, Conjecture~\ref{c:1}
is optimal for the most interesting range of $\frac sn$. 

\begin{fact}\label{fact:1744}
	If $s\in\bigl(\frac13 n, \frac 12 n\bigr)$, then $\ex\ge g(n, s)$.
\end{fact}

\begin{proof}
	Let $k\ge 2$ be the unique integer with 
	$s\in\bigl[\frac{k}{3k-1}n, \frac{k-1}{3k-4}n\bigr)$.
	Take a blow-up $G$ of $\Gamma_k$ obtained by replacing the vertices $1$, $k$, and $2k$
	by sets of size $(k-1)n-(3k-4)s$ and the remaining vertices by sets of size $3s-n$.
	Clearly, $G$ is triangle-free. One can check that $G$ has $n$ vertices, 
	independence number $s$, and $g_k(n, s)$ edges. 
\end{proof}

Let us observe that Conjecture~\ref{c:1} yields a precise prediction on the 
Ramsey-Tur\'an density function $f$, namely the following.

\begin{conj}\label{conj:2}
	The function $f\colon (0, 1]\longrightarrow \RR$ is given by
	\[
		f(\alpha)=
		\begin{cases}
			\alpha\,, & \text{ if $\alpha\le \frac13$} \cr
					f_k(\alpha)\,, & 
					\text{ if $\frac{k}{3k-1} \le \alpha < \frac{k-1}{3k-4}$ for some $k\ge 2$} \cr
			\frac 12, & \text{ if $\frac 12\le \alpha\le 1$},
		\end{cases}
	\]
	where 
	\begin{equation}\label{eq:1801}
		f_k(\alpha)=k(k-1)-2k(3k-4)\alpha+(3k-4)(3k-1)\alpha^2\,.
	\end{equation}
\end{conj}

Notice that at the critical values $\alpha_k=\frac k{3k-1}$ this function 
agrees with~\eqref{eq:1803}. The remainder of this introduction discusses 
further evidence in support of Conjecture~\ref{c:1}. 

\subsection{Minimum degree}

There appears to be a mysterious analogy between the Ramsey-Tur\'an problem for 
triangles and the more thoroughly studied problem to describe the structure of
triangle-free graphs of large minimum degree. For instance, there is a similar 
transition from chaos to structure occurring at $\frac 13n$. As reported in~\cite{ES73} 
Hajnal constructed triangle-free graphs $G$ with 
$\delta(G)\ge \bigl(\frac 13-o(1)\bigr)|V(G)|$
of arbitrarily large chromatic number. On the other hand, \L uczak~\cite{Lu06} 
proved that for every $\eps>0$
all triangle-free graphs $G$ with $\delta(G)\ge \bigl(\frac13+\eps\bigr)|V(G)|$ 
admit a homomorphism into a triangle-free graph whose order can be bounded in terms 
of $\eps$. The ultimate variant of \L uczak's result is due to 
Brandt and Thomass\'e~\cite{BT}, who proved that, actually, such graphs either admit 
a homomorphism into some Andr\'asfai graph or into some 
so-called {\sl Vega graph} (see~\cite{BP}).

We will not introduce Vega graphs properly here and only remark that they can be 
obtained from Andr\'asfai graph by adding $8$ vertices forming a cube as well as several 
new edges, and possibly deleting at most two special vertices afterwards. 
For $k\ge 10$ with $k\equiv 1, 2, 3\pmod{9}$ there are blow-ups of appropriate 
Vega graphs with $3k-1$ vertices and independence number $k$ that furnish additional 
extremal cases for~\eqref{eq:1803}.      

Now several questions present themselves. Are Vega graphs so special that their 
only appearances in extremal families $\Ex$ are the aforementioned regular ones?
Or, perhaps, on the contrary, because of their more sophisticated structure they 
admit blow-ups falsifying Conjecture~\ref{c:1}? Both these speculations seem to be false. 
There are non-trivial blow-ups of Vega graphs which are in $\Ex$. Nonetheless, it can be 
proved that, in a sense made precise in~\cite{Vega}, no `natural' blow-up of an Andr\'asfai 
graph or Vega graph 
can be a counterexample to Conjecture~\ref{c:1}.

These facts have an interesting consequence. Together with the structure theorem of 
Brandt and Thomass\'e they allow us to prove Conjecture~\ref{c:1} for  
$\alpha\in \bigl[\frac k{3k-1}, \frac k{3k-1}+\eps_k\bigr]$, where $k\ge 2$ and $\eps_k$ 
is sufficiently small.
Further details and a conjectural explicit description of the extremal graph 
families $\Ex$ will be presented in our forthcoming article~\cite{Vega}. 

\bigskip

This article is organized as follows. In the next section we sketch some tools 
and observations which we shall use later and which, hopefully, could be useful 
in the quest of proving Conjecture~\ref{c:1} in full generality. The last section 
is devoted to the proof of Theorem~\ref{thm:2338}.   


\section{Preliminaries}

The goal of this section is to gather several results that we believe to be 
relevant in general to the problem of determining $\ex$ for $s>n/3$. 
These preliminaries fall naturally into three groups. 
We start in Subsection~\ref{subsec:MI} with some facts concerning matchings 
and independent sets in arbitrary, not necessarily triangle-free, graphs. 
Subsection~\ref{subsec:Zykov} proceeds with a discussion of {\it 
symmetrisation} operations -- a device we shall use for ``simplifying'' extremal
graphs. Finally in Subsection~\ref{subsec:AB} we use this technique for investigating
the structure of graphs that are extremal for the problem to determine~$\ex$.   
  
Throughout the article we follow standard graph theoretical notation. 
Thus, for instance,~$\deg_G(v)$ stands for the degree of a vertex $v$ of a graph $G$, 
and by $\Nn_G(S)$ we mean the neighbourhood of the set of vertices $S$. Moreover, 
we omit subscripts unless they are necessary to avoid confusion. Given two disjoint 
sets $A$ and $B$ we define 
\[
	K(A, B) = \bigl\{\{a, b\}\colon a\in A \text{ and } b\in B\bigr\}\,.
\]
This is the edge set of the complete bipartite graph with vertex partition $A\dcup B$.
For two disjoint subsets $A$ and $B$ of vertices of $G$ we say that {\it $A$ is 
matchable into~$B$} 
if the bipartite graph induced in $G$ by the sets $A$ and $B$ contains  
a matching saturating $A$. 

\subsection{Matchings and independent sets} \label{subsec:MI}

Clearly if $A$ and $Y$ are two disjoint independent sets in a graph $G$ with 
$|Y|\le |A|$, then there is an injective map from $Y$ to $A$. The following simple 
consequence of Hall's theorem, due to Andr\'asfai (see~\cite{A}*{Lemma~2.3}), ensures 
that in case $|A|=\alpha(G)$
one such injection is exemplified by a matching. For the reader's 
convenience we include a short proof.

\begin{fact}\label{f:matchable}
	Let $A$ and $Y$ be two disjoint independent sets in a graph $G$. If $|A|=\alpha(G)$, 
	then~$Y$ is matchable into $A$. 
\end{fact}

\begin{proof}
	In the light of Hall's theorem~\cite{Hall} it suffices to 
	prove that for an arbitrary $D\subseteq Y$ and its neighbourhood $A\cap \Nn(D)$ in $A$ 
	we have $|D|\le |A\cap \Nn(D)|$. Since $D\cup (A\setminus \Nn(D))$ is independent in $G$, 
	we have indeed 	
	\[
		|D| + |A| - |A\cap \Nn(D)| = |D\cup (A\setminus \Nn(D))| \le \alpha(G) = |A|\,. \qedhere
	\]
\end{proof}

In general, deleting edges from a graph may increase its independence number. 
For our purposes it will be important to know that the following type
of edge deletions leave the independence number invariant. 

\begin{lemma}\label{f:isolating}
Given a graph $G$, suppose
	\begin{enumerate}
		\item[$\bullet$]  that $A\subseteq V(G)$ is an independent set of size $\alpha(G)$, 
		\item[$\bullet$] and that $M$ is a matching in $G$ from $V(G)\setminus A$ to $A$, 
			the size of which is as large as possible. 
	\end{enumerate}
	If $G'$ denotes the graph obtained from $G$ by isolating the vertices in 
	$A\setminus V(M)$, i.e., by deleting all edges incident with them, then 
	$\alpha (G')=\alpha(G)$. 
\end{lemma}

\begin{proof}
	Since $G'$ is a subgraph of $G$, we have  $\alpha (G')\ge \alpha(G)$. 
	For the converse direction we consider any set $U\subseteq V(G)$ which is 
	independent in $G'$. 
	Since $U\setminus A$ is independent in $G$, Fact \ref{f:matchable} tells us that 
	in $G$ there exists a matching $N$ from $U\setminus A$ to $A$ covering all vertices 
	of $U\setminus A$. We contend that 
	\[
		M' = \{e\in M\colon A\cap U\cap e\neq \varnothing\} \cup N
	\]
	is a matching in $G$. Otherwise, there had to exist two edges, $e\in M$ with 
	$A\cap U\cap e\neq \varnothing$ and $f\in N$, sharing a vertex $x$. By $x\in f\in N$ 
	we have either $x\in U\sm A$ or $x\in A$. 
	In the former case the vertex of $e$ distinct from $x$ needs to be in $A\cap U$. 
	In particular, both ends of $e$ are in $U$, contrary to $e\in M \subseteq E(G')$ 
	and $U$ being independent in $G'$.
	
	So we are left with the case $x\in A$. Now we have in fact $A\cap U \cap e = \{x\}$ 
	and both ends
	of $f$ are in $U$. Moreover, $x\in A\cap U\cap e\subseteq A\cap V(M)$ shows that $f$ does 
	not get deleted when we pass from $G$ to $G'$. This contradiction to $U$ being
	independent in $G'$ proves that~$M'$ is indeed a matching in $G$.  
		
	Therefore the maximality of $M$ yields
	\[
		|N| + |V(M)\cap A\cap U|=|M'| \le |M| 
		= |V(M)\cap A\cap U| + |V(M) \cap (A\setminus U)|\,,
	\]
	i.e.,
		\[
		|U\setminus A|=|N| \le |V(M)\cap (A\setminus U)|\,.
		\]
	Thus 
		\[
			|U\setminus A|\le |A\setminus U|\,,
		\]
and, consequently,  $|U|\le |A|=\alpha(G)$, as desired. 
\end{proof}

To unleash the full power of the foregoing lemma it is useful to know 
that certain subsets of $V(G)$ can be forced to be
subsets of $V(M)$. For such purposes we shall employ the following observation. 
  
\begin{fact}\label{f:bipartite}
	Let $H$ be a bipartite graph with vertex classes $R$ and $S$ in which the largest 
	matching has size $m$. If some set $R'\subseteq R$ is matchable into $S$, 
	then $H$ contains a matching of size $m$ saturating all vertices of $R'$.
\end{fact}

\begin{proof}
	Let $\ccM$ be the set of all matchings in $H$ having size $m$. 
	Choose first a matching~$N$ from $R'$ into $S$ saturating all vertices in $R'$, 
	and then a matching 
	$M\in \ccM$ for which $|M\cap N|$ is maximal. We will show that $M$ 
	covers $R'$. Otherwise there was a vertex $x\in R'$ not covered by $M$. 
	Let $z$ be the unique vertex with $xz\in N$. By the maximality of $m$, 
	the $m+1$ edges in $M\cup \{xz\}$ cannot form a matching in $H$
	and, consequently, we have $z\in V(M)$. So there is a vertex~$u$ such 
	that $uz\in M$. But now $M' = (M\setminus \{uz\})\cup \{xz\}$
	is a matching in $\ccM$ satisfying $|M'\cap N|>|M\cap N|$,  contradicting 
	our choice of $M$.
\end{proof}

\subsection{
	Symmetrisation} \label{subsec:Zykov}

One of the standard proofs of Tur\'an's theorem~\cite{T}, due to 
Zykov~\cite{Zy}, shows that every $K_\ell$-free graph $G$ can be 
transformed into an $(\ell-1)$-partite graph
on the same vertex set 
by a sequence of symmetrisation operations,  in such a way that throughout 
the whole process the number of edges never decreases. Indeed, this implies that $G$ does 
have at most as many edges as the corresponding Tur\'an graph. The  
{\it symmetrisation} step employed  by Zykov consists in taking two non-adjacent 
vertices $u$ and~$v$, deleting all edges incident with $u$, and then adding all edges 
from $u$ to the neighbours of~$v$. 

Our proof of Theorem~\ref{thm:2338} utilises a modest generalisation of this idea.
Given a graph~$G$  
and two disjoint sets $A,B\subseteq V(G)$, we say that 
a graph~$G'$ on the same vertex set as~$G$ arises from~$G$ by the {\it generalised Zykov 
symmetrisation} $\GZS(A, B)$ if it is obtained by deleting all edges incident with $B$
and afterwards adding all edges from $A$ to~$B$. Explicitly, this means 
\[
	 V(G')=V(G) 
	 \quad \text{ and } \quad 
	 E(G') = \bigl(E(G)\setminus \{e\in E(G): e\cap B\neq\emptyset\}\bigr)\cup K(A,B)\,.
\]
We will express this state of affairs by writing $G'=\GZS(G \,|\, A, B)$.
In the special case where $B=\{v\}$ is a singleton, we will often
abbreviate $\{v\}$ to $v$, thus speaking, e.g. of the operation $\GZS(A, v)$. 
For later use we record the following obvious properties of these operations.

\begin{fact}\label{f:NRO}
	Given a graph $G$ and two disjoint sets $A,B\subseteq V(G)$, let $G'=\GZS(G \,|\, A, B)$.
	\begin{enumerate}[label=\rmlabel]
		\item\label{it:11} If $|A|\ge \deg_G(b)$ holds for all $b\in B$, then $e(G') \ge e(G)$.
		If equality occurs, then all vertices in $B$ have degree $|A|$ in $G$
		and $B$ is an independent set in $G$. 
		\item\label{it:22} If $A$ is independent and $G$ is triangle-free, then so is $G'$.
	\end{enumerate}
\end{fact}

\begin{proof}
	Part~\ref{it:11} follows from the estimate
	\[
		e(G') - e(G)
		\ge 
		\sum_{b\in B} \bigl(|A| -\deg_G(b)\bigr) 
		\ge
		0\,,
	\]
	where the first ``$\ge$'' sign takes into account that edges both of whose 
	ends are in $B$ are subtracted twice in the sum over $b\in B$. The statement 
	addressing the equality case should now be clear.
	
	To prove part~\ref{it:22} we assume 
	for the sake of contradiction that $xyz$ was a triangle in~$G'$. Owing to 
	$G-B=G'-B$ we may suppose further that $x\in B$. Now $y$ and $z$ are neighbours 
	of $x$ in~$G'$ and, hence, both of them are in~$A$. But $A$ is still independent 
	in $G'$ and thus~$yz$ cannot be an edge of $G'$.
\end{proof}

Our next result describes a case where symmetrization preserves the independence number. 

\begin{lemma}\label{f:NROs}
	Let $A$ and $B$ be two disjoint independent sets in a graph $G$ such 
	that $|A|=|B|=\alpha(G)$. If $M$ is a matching from $V(G)\setminus (A\cup B)$ 
	to $B$, whose size is as large as possible, $B'\subseteq B\setminus V(M)$, 
	and $G'=\GZS(G\,|\, A, B')$, then $\alpha(G')=\alpha(G)$. 
\end{lemma}

\begin{proof}
	Since  $A$ is independent in $G'$, we have $\alpha(G')\ge\alpha(G)$. 
	Now suppose conversely that $U\subseteq V(G)$ is independent in $G'$. 
	We are to prove $|U|\le \alpha(G)$. 
	
	If $A$ and $U$ are disjoint it follows from Lemma~\ref{f:isolating}  (applied 
	to $G-A$, $B$, and $M$ here in place of $G$, $A$, and $M$ 
	there) that $|U|\le \al(G-A)= \alpha(G)$, meaning that we are done. 
 
	It remains to consider the case~$U\cap A\neq \emptyset$. 
	Now $K(A, B')\subseteq E(G')$ implies that~$U$ is disjoint 
	to $B'$. Therefore $U$ is independent in $G$ and so $|U|\le \al(G)$. 
\end{proof}	

\subsection{Some general results} \label{subsec:AB}

Recall that for $n \ge s \ge 0$ we are interested in the quantity
\begin{align*}
		\ex=\max\bigl\{|E|\colon G=(V, E) 
			\text{ is a triangle-free graph with $|V|=n$ and 
			$\al(G)\le s$}\bigr\}
\end{align*}
and that
\[
	\Ex
	=
	\bigl\{G=(V,E)\colon K_3\nsubseteq G,\;|V|=n, \;\al(G)\le s,\textrm{ and } |E|=\ex\bigr\}
\]
denotes the corresponding family of extremal graphs. 

We begin by observing that the estimate $\ex\le g_k(n, s)$
holds whenever $s$ is outside the range required by Conjecture~\ref{c:1}.

\begin{fact}\label{f:kn<s<(k-1)n}
	Let integers $n\ge s \ge 0$ and $k\ge 2$ be given. 
	If $s\notin \bigl(\frac{k}{3k-1}n, \frac{k-1}{3k-4}n\bigr)$, then 
	\[
		\ex\le g_k(n, s)
	\]
	and equality can only hold if $s\in\bigl\{\frac{k}{3k-1}n, \frac{k-1}{3k-4}n\bigr\}$.
\end{fact}
\begin{proof}
	We check that under our assumption on $s$ the trivial 
	upper bound $\ex\le \frac12 ns$ is at least as good as $g_k(n, s)$. Since
	for  $s\notin \bigl(\frac{k}{3k-1}n, \frac{k-1}{3k-4}n\bigr)$,
	\begin{align*}
		ns & \le ns   + \bigl(kn-(3k-1)s\bigr)\bigl((k-1)n-(3k-4)s\bigr) \\
		 & =k(k-1)n^2-2k(3k-4)ns+(3k-4)(3k-1)s^2 
		 =2g_k(n, s)
	\end{align*}
	this is indeed the case and the statement about the equality case also easily follows.
\end{proof}

In combination with Fact~\ref{fact:1744} this leads to the following alternative 
way of defining~$g(n, s)$ in case $\frac sn\in \bigl(\frac 13, \frac12\bigr)$.

\begin{cor}
	If $\frac 13n < s < \frac 12n$, then 
	\begin{align*}\tag*{$\Box$}
		\hfill g(n, s)=\min\bigl\{g_k(n, s)\colon k\ge 2\bigr\}\,.   
	\end{align*}
\end{cor} 

Our first structural result on graphs in $\Ex$
asserts that they contain two disjoint independent sets of size $s$ 
(provided there is enough space for them).

\begin{lemma}\label{f:ABs}
	If $n$ and $s$ are two integers with $n\ge 2s \ge 0$, then every 
	graph $G\in \Ex$ contains two disjoint independent sets of size $s$.  
\end{lemma}

\begin{proof}
	Let $(X, Y)$ be a pair of disjoint independent sets in $G$ such that 
	$|X|=\alpha(G)$ and subject to this $|Y|$ is as large as possible. 
	Clearly $|Y|\le |X|\le s$	and we are to prove that equality holds throughout. 
	This could fail in two different ways. 
	
	\smallskip

	{\it  \hskip1em   Case 1.\quad $|Y|<|X|$} 

	\medskip

	Owing to $\alpha(G-X)=|Y|\le |X|-1\le s-1<n-s\le v(G-X)$ there is 
	an edge~$ab$ of~$G$ with $a, b\not\in X$. If both of $a$ and $b$ have degree 
	$\alpha(G)$, then their neighbourhoods are two disjoint independent sets of 
	size $\alpha(G)$, thus contradicting our choice of the pair~$(X, Y)$. 
	
	It follows that we may assume, without loss of generality, that $\deg(a)< \alpha(G)$.
	By Fact~\ref{f:NRO} the graph $G'=\GZS(G\,|\, X, a)$ is triangle-free and has more 
	edges than $G$. So the extremality of $G$ entails that $G'$ contains an independent 
	set of size $\alpha(G)+1$. Owing to $G-a=G'-a$ any such set needs to be of the 
	form $Z\cup\{a\}$, where $Z$ is an independent set in $G$ of size~$\alpha(G)$. Due to the 
	construction of $G'$ the sets $X$ and $Z$ need to be disjoint and thus the pair 
	$(X, Z)$ contradicts our choice of $(X, Y)$.   
	
	\smallskip

	{\it \hskip1em Case 2.\quad $|X| = |Y| = \alpha(G) < s$} 

	\medskip

	Now $|V(G)\setminus (X\cup Y)|\ge n-2(s-1)\ge 2$ and thus there are two distinct 
	vertices in this set, say $a$ and $b$. Let $G'$ be the result of applying 
	first $\GZS(Y\cup \{b\}, a)$ and then $\GZS(X \cup \{a\}, b)$ to $G$ (see
	Figure~\ref{fig:XYab}).
	
	\begin{figure}[ht]
			\centering
				\begin{tikzpicture}[scale=0.62]
				
				\coordinate (a) at (7,2);
				\coordinate (b) at (7,-2);
				
				\coordinate (x1) at (-1,2);
				\coordinate (x2) at (1,2);
				\coordinate (x3) at (3,2);
				\coordinate (x4) at (5,2);
			
				\coordinate (y1) at (-1,-2);
				\coordinate (y2) at (1,-2);
				\coordinate (y3) at (3,-2);
				\coordinate (y4) at (5,-2);
					
				\begin{pgfonlayer}{front}
				\foreach \i in {1,...,4}{
					\draw [green!75!black, very thick]  (a) -- (y\i) ;
					\draw [green!75!black, very thick]  (b) -- (x\i) ;
				}	
				\draw [green!75!black, very thick]  (a) -- (b) ;
				\foreach \i in {1,...,4}{
					\draw[blue!75!black, very thick]  (y\i) circle (2pt);
					\fill[blue!75!white]  (y\i) circle (2pt);
					\draw[blue!75!black, very thick]  (x\i) circle (2pt);
					\fill[blue!75!white]  (x\i) circle (2pt);
				}	
					\draw[red!75!black, very thick]  (a) circle (2pt);
					\fill[red!75!white]  (a) circle (2pt);
					\draw[red!75!black, very thick]  (b) circle (2pt);
					\fill[red!75!white]  (b) circle (2pt);
				\end{pgfonlayer}
				
				\draw[blue!75!black, line width=2pt] (2,2) ellipse (4.2cm and 18pt);
				\fill[blue!75!white,opacity=0.2] (2,2) ellipse (4.2cm and 18pt);
				
				\draw[blue!75!black, line width=2pt] (2,-2) ellipse (4.2cm and 18pt);
				\fill[blue!75!white,opacity=0.2] (2,-2) ellipse (4.2cm and 18pt);
				
				\node at (2.3,2) {\large\textcolor{blue!75!black}{$X$}};
				\node at (2.3,-2) {\large\textcolor{blue!75!black}{$Y$}};
				
				\node at (a)[above, right] {\large\textcolor{red!75!black}{$a$}};
				\node at (b) [below, right]{\large\textcolor{red!75!black}{$b$}};
					
				\end{tikzpicture}
				\caption{Possibly new edges of $G'$ are drawn green.}
				\label{fig:XYab}
			\end{figure}
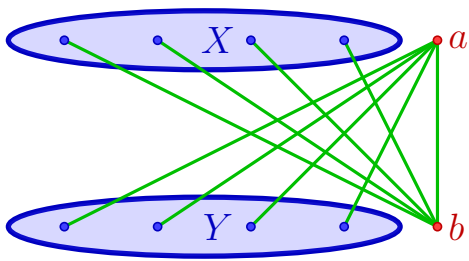
	
	\noindent
	Observe that the  
	sets $Y\cup\{b\}$ and $X\cup\{a\}$ are independent in $G'$, whence 
	$G'$ is triangle-free. Since $ab$ is an edge of $G'$, 
	we have $\alpha(G')\le \alpha(G)+1\le s$. 
	Notice that for a vertex $v$ of $G$ we have $\deg_G(v)\le \alpha (G)=|X|=|Y|$ and so 
	\[
		e(G') \ge  e(G) - \deg_G(a)-\deg_G(b)+|X|+|Y|+1 > e(G)\,,
	\]
	which clearly contradicts the fact that $G\in \Ex$.
\end{proof}
	
We conclude this section with a result that provides additional information on 
the structure of some graphs in $\Ex$. 
	
\begin{lemma} \label{f:twosets}
	Given two integers $n\ge 0$ and $s\in \bigl[\frac13 n, \frac12 n\bigr]$, there
	exists a graph $G\in \Ex$ containing two disjoint independent sets $A$ 
	and $B$ of size $|A|=|B|=s$ having subsets $A'\subseteq B$ and $B'\subseteq A$ 
	with $|A'|=|B'|=3s-n$ and 
\begin{equation}\label{eqll}
		K(A',A)\cup K(B',B) \subseteq E(G)\,.
\end{equation}
\end{lemma}
	
\begin{proof}
	Let $G'\in\Ex$ and
	let $A, B\subseteq V(G')$ be any two disjoint independent sets of size $s$, the existence 
	of which is guaranteed by Lemma~\ref{f:ABs}. Set $X=V(G')\setminus (A\cup B)$ and 
	denote by $M_A$ and $M_B$ maximum matchings in $G'$ from $X$ to $A$ and from $X$ 
	to $B$, respectively. 
	
	Since $|X|=n-2s$ we have $|M_A|, |M_B|\le n-2s$, wherefore the sets 
	$A_*=B\setminus V(M_B)$ and $B_*=A\setminus V(M_A)$ 
	satisfy $|A_*|, |B_*|\ge s-(n-2s) = 3s-n$. Take arbitrary 
	subsets $A'\subseteq A_*$, $B'\subseteq B_*$ of size $3s-n$, and let $G$ 
	denote the graph arising from $G'$ 
	by applying first~$\GZS(A, A')$ and then $\GZS(B, B')$. 
	Two successive applications of Fact~\ref{f:NRO} and Lemma~\ref{f:NROs}
	tell us that   $G$ is  triangle-free,
	$\al(G)=s$, and $e(G)\ge e(G')$. Thus, we arrive at a graph
	$G\in\Ex$ for which (\ref{eqll}) holds.
\end{proof}

	
\section{The proof of Theorem~\ref{thm:2338}}

\subsection{Andr\'asfai's result} \label{subsec:31}

Let us recall that $g_2(n, s)=n^2-4ns+5s^2$. The main result of this subsection,
Proposition~\ref{f:k2strong} below, asserts that this expression is an upper bound on
the number of edges of triangle-free $n$-vertex graphs satisfying a less restrictive 
condition than $\alpha(G)\le s$. Therefore, this result provides a technical strengthening
of Andr\'asfai's theorem quoted in the introduction. Our reason for dealing 
with such a statement here is that we refer to it in the proof that $\ex\le g_3(n, s)$.   
We start with the following observation. 

\begin{lemma} \label{lem:g2}
	Let $n$ and $s$ be positive integers with $s\in\bigl[\frac 13n, \frac 12 n\bigr]$
	and suppose that $G$ is a triangle-free graph on $n$ vertices containing 
	two disjoint independent sets $A$ and $A'$ with $|A|=s$ and $|A'|\ge 3s-n$. If 
	\begin{enumerate}[label=\rmlabel]
		\item\label{it:31i} $\deg(a)\le s$ for all $a\in A$,
	\end{enumerate}
and
\begin{enumerate}[label=\rmlabel, resume]
		\item\label{it:31ii} $\Nn(a')=A$ for all $a'\in A'$,
	\end{enumerate}
	then $e(G)\le n^2-4ns+5s^2$. 
\end{lemma} 

\begin{proof}
	By passing to a subset if necessary we may assume $|A'| = 3s-n$. 
	The graph $G-(A\cup A')$ has $n-s-(3s-n)=2(n-2s)$ vertices, so by Mantel's
	theorem it has at most $(n-2s)^2$ edges. By~\ref{it:31ii} all edges of $G-A$ 
	are actually edges of $G-(A\cup A')$ and thus we have 
	\[
		e(G)\le \sum_{a\in A}\deg(a) +  (n-2s)^2\,.
	\]
	Owing to~\ref{it:31i} this leads to $e(G)\le s^2+(n-2s)^2=n^2-4ns+5s^2$,
	as desired.
\end{proof}

We remark that the tools we have 
developed so far lead to a short proof of Andr\'asfai's main result in~\cite{A}.

\begin{proof}[Proof of Theorem~\ref{thm:1730}]
	Due to Lemma~\ref{f:twosets} there exists a graph
	$G\in \Ex$ containing two disjoint independent sets $A$ and $A'$
	with 
	\begin{equation} \label{eq:AAA}
		|A|=s, \quad |A'|=3s-n, \quad \text{and} \quad K(A',A)\subseteq E(G)\,.
	\end{equation}

	Recall that the absence of triangles in $G$ implies $\Delta(G)\le \alpha(G)\le s$. Thus $G$ 
	and the sets~$A$,~$A'$ satisfy the hypothesis of Lemma~\ref{lem:g2} and, consequently, 
   \[
   	\ex=e(G)\le n^2-4ns+5s^2\,.
	\]
   On the other hand, Fact~\ref{fact:1744} shows $\ex\ge n^2-4ns+5s^2$.
   (See also the blow-up of the pentagon presented in Figure~\ref{fig:pentagon}.)
\end{proof}

The main result of this subsection is a generalisation
of Theorem~\ref{thm:1730} that allows vertices whose degree is larger than $s$.
On the other hand, if one applies the statement that follows to a triangle free graph $G$ 
with $\alpha(G)=s\le \tfrac12n$, then $Z=\varnothing$ and an arbitrary choice of $Q$ leads 
to the estimate $e(G) \le n^2 - 4ns + 5s^2$. 	

\begin{prop}\label{f:k2strong}
	For $s\in\bigl[\frac13 n, \frac12 n\bigr]$, 
	let $G$ be a triangle-free graph on $n$ vertices 
	containing an independent set $A$ of size $s$, for which $\alpha(G-A) \le s$.
	Let
	\[
		Z = \{ v \in V(G)\setminus A\colon \deg(v) > s\}
	\]
	and let  $Q\subseteq V(G)\setminus (A\cup Z)$ be any set with $|Q| \ge 3s-n$.
	If every independent set $Z'\subseteq Z$ is matchable into 
	$V(G)\setminus (A\cup Z'\cup Q)$, then 
	\[
		e(G) \le n^2 - 4ns + 5s^2\,.
	\]
\end{prop}
\begin{proof}
	Let $n$ and $s$ be fixed and assume for the sake of contradiction, that there 
	exists a counterexample, i.e., a triple $(G, A, Q)$ satisfying all the assumptions 
	of Proposition~\ref{f:k2strong}, but for which $G$ has more than $n^2-4ns+5s^2$ edges. 
	Among all such counterexamples we choose one for which the size of the set
	\[
		Q'=\{q\in Q\colon \Nn(q)=A\}
	\]
	is as large as possible. Observe that Lemma~\ref{lem:g2} applied to $Q'$ here in place of $A'$ 
	there yields $|Q'|<3s-n$, whence 
	\begin{equation}\label{eq:Q}
		Q'\neq Q\,.
	\end{equation}
	
	Next we use the maximality of $Q'$ for showing that the 
	assumption $\alpha(G-A) \le s$ holds with equality. 
	
	\goodbreak
	
	\begin{claim}\label{clm:1917}
		There is an independent set $B\subseteq V(G)\setminus A$ of size $s$.
	\end{claim}
	
	\begin{proof}
		Assume $\alpha(G-A)\le s-1$.
		Owing to~\eqref{eq:Q} we may pick a vertex $q_*\in Q\setminus Q'$ and apply 
		$\GZS(A, q^*)$ to $G$, thus getting a graph $G^*$. By Fact~\ref{f:NRO} 
		and $q^*\not\in Z$ we know that $G^*$ is triangle-free and has at least as many 
		edges as $G$. Clearly, $\al(G^*-A)\le \al(G-A) + 1 \le s$ and 
		the assumption of Proposition~\ref{f:k2strong} holds for $G^*$ and  the 
		sets $A$, $Q$, 
		and 
		\[
			Z^*= \{ v^* \in V(G^*)\setminus A\colon \deg_{G^*}(v^*) > s\}\subseteq Z\,.
		\]
		Since $|Q'\cup \{q^*\}|>|Q'|$, the  maximality of $|Q'|$ implies
		that $(G^*, A, Q)$ cannot be a counterexample to our result, which yields
		\[
			e(G)\le e(G^*) \le  n^2-4ns + 5s^2\,, 
		\]
		contrary to the choice of $(G, A, Q)$. Thereby Claim~\ref{clm:1917} is proved.
	\end{proof}
	
	Working with the set $B$ obtained in the previous claim we define $X=V(G)\sm (A\cup B)$  
	and $Z'=Z\cap B$. Due to the independence of $B$ and our hypothesis $Z'$ is matchable 
	into $X$. So Fact~\ref{f:bipartite} applied to $Z'$, $B$, and $X$ here in place of
	$R'$, $R$, and $S$ there yields a maximum matching $M$ between $B$ and $X$ which 
	covers $Z'$. 
	
	Now we set $B'=B\setminus V(M)$ and look at the graph $G'=\GZS(G\,|\, A, B')$. 
	Note that 
	$$|B'|\ge s-|X|=s-(n-2s)=3s-n\,.$$
	Fact~\ref{f:NRO} reveals that $G'$ is triangle-free and satisfies $e(G')\ge e(G)$.
	Lemma~\ref{f:isolating} applied to $G-A$, $B$, and~$M$ reveals 
	$\al(G'-A)= \alpha(G-A)=s$ and, consequently, every $a\in A$ has at most $s$ 
	neighbours in $G'$. Thus $G'$, $A$, and $B'$ here in place of $G$, $A$, 
	and $A'$ there satisfy the assumptions of Lemma~\ref{lem:g2}, meaning that   
	\[
		e(G) \le e(G')\le n^2-4ns+5s^2\,.
	\]
	This contradiction to $(G, A, Q)$ being a counterexample establishes 
	Proposition~\ref{f:k2strong}.
\end{proof}
	
\subsection{Blow-ups of Wagner graphs}
The present subsection completes the proof of Theorem~\ref{thm:2338},
asserting that every graph $G\in\Ex$ satisfies 
$e(G)\le 3n^2 - 15 ns +20 s^2$. Recall that by Lemma~\ref{f:ABs} any such graph $G$ 
contains two disjoint independent sets $A$ and~$B$ of size $s$. Our next result shows 
that if there exists a further independent set of size~$s$ having appropriate intersections 
with $A$ and $B$, then we can reach our goal.  
	
\begin{lemma}\label{f:3sets}
	Given $n\ge s \ge 0$ let $G$ be a triangle-free graph on $n$ vertices with 
	$\alpha (G) = s$. If $G$ contains three independent sets $A$, $B$, and $C$ 
	of size $s$ such that $A\cap B, A\cap C= \emptyset$ and $|B\cap C|\le n-2s$,
	then 
	\[
		e(G) \le 3n^2 - 15 ns +20 s^2\,.
	\]
\end{lemma}
\begin{proof}
	Recall that by the case $k=3$ of Fact~\ref{f:kn<s<(k-1)n} we may assume 
	$\frac 38n<s< \frac 25 n$. 
	Our argument is reminiscent of the proof of 
	Lemma~\ref{f:twosets}. Fix three independent sets~$A$,~$B$, and~$C$ in~$G$ 
	such that $A$ is disjoint to $B$, $C$ and $|B\cap C|\le n-2s$.  
	
	\begin{claim} \label{clm:37}
	We may assume that there are sets $B'\subseteq A$, $C'\subseteq A\setminus B'$,
	and $A'\subseteq B\cap C$ of size $|A'|=|B'|=|C'|=3s-n$ such that 
	\[
		K(A',A) \cup K(B',B) \cup K(C',C) \subseteq E(G)\,.
	\]
	\end{claim} 
	
	\begin{proof}	
		Take a maximum matching $M_B$ from $V(G)\setminus (A\cup B)$ into $A$ and 
		consider the graph $G_1=\GZS\bigl(G\,|\, B, A\setminus V(M_B)\bigr)$. 
		By Fact~\ref{f:NRO} this graph is triangle-free and satisfies $e(G_1)\ge e(G)$. 
		Lemma~\ref{f:NROs} yields $\alpha(G_1)=s$ and one checks easily 
		that $A$, $B$, and $C$ are still independent in $G_1$. Moreover, 
		\[
			|A\setminus V(M_B)|\ge |A|-|V(G)\setminus (A\cup B)| =s-(n-2s)=3s-n>0
		\]
		%
		so there is a set $B'\subseteq A\setminus V(M_B)$
		with $|B'|=3s-n$ and clearly we have $K(B',B) \subseteq E(G_1)$.   
		
		Next we observe that the assumption $|B\cap C|\le n-2s$ entails 
		$|B\setminus C|\ge 3s-n=|B'|$ and therefore in $G_1$ the set $B'$ is matchable 
		into $B\setminus C$, which is a subset of $V(G)\setminus (A\cup C)$. 
		So by Fact~\ref{f:bipartite} there is a maximum matching $M_C$ from 
		$V(G)\setminus (A\cup C)$ to $A$ which covers all vertices in $B'$.
		As in the previous paragraph one proves that the graph 
		$G_2=\GZS(G_1\,|\, C, A\setminus V(M_C))$ is triangle-free, has independence number $s$
		and at least as many edges as $G_1$. Also, as before one finds a set 
		$C'\subseteq A\setminus V(M_C)$ with $|C'|=3s-n$ and observes 
		$K(C',C) \subseteq E(G_2)$. The sets $A$, $B$, and $C$ are 
		still independent in $G_2$ and our reason for insisting on $B'\subseteq V(M_C)$
		was that it ensures $C'\subseteq A\sm B'$.
	
		Finally, we let $M_A$ be a maximum matching in $G_2$ from $V(G)\setminus (A\cup B)$ 
		to $B$ and put $G_3=\GZS\bigl(G_2\,|\, A, B\setminus V(M_A)\bigr)$. Standard 
		arguments show that $G_3$ is triangle-free and satisfies $\alpha(G_3)=s$ as well as 
		$e(G_3)\ge e(G_2)$. Furthermore, there is a set $A'\subseteq B\setminus V(M_A)$  
		with $|A'|=3s-n$  such that $K(A',A) \subseteq E(G_3)$.
		Moreover, since $C\cup A'$ is an independent set, we have $A'\subseteq C$.
		Altogether, the graph $G_3$ has all desired properties and owing to 
		$e(G_3)\ge e(G)$ we may continue the proof with $G_3$ instead of $G$.
		This proves Claim~\ref{clm:37}.
	\end{proof}

	Now we set $n^*=4n-9s$, $s^*=n-2s$,  
	$G^*=G-(A'\cup B'\cup C')$, and $A^* = B\setminus A'$, $Q^* = C\setminus B$.
	
	\begin{claim}\label{clm:1913}
		The numbers $n^*$ and $s^*$ as well as the triple $(G^*, A^*, Q^*)$ 
		satisfy the hypothesis of Proposition~\ref{f:k2strong}. 
	\end{claim}
	
	\begin{proof}
		A quick calculation based on $s\in \bigl[\frac 13 n, \frac 25 n\bigr]$
		shows that $n^*>0$ and $s^*\in\bigl[\frac 13 n^*, \frac12 n^*]$, i.e., that the 
		size of $s^*$ is in the appropriate range. 
		Owing to
		\[
			|V(G^*)|=n-(|A'|+|B'|+|C'|)=n-3(3s-n)=4n-9s=n^*
		\]
		and $|A^*|=|B|-|A'|=s-(3s-n)=n-2s=s^*$ the sets $V(G^*)$ and $A^*$
		have the correct size. Next, we would like to show $\alpha(G^*-A^*)\le s^*$.
		If $J\subseteq V(G^*)\setminus A^*$ is independent in~$G^*$, then $J\cup B'$ 
		is independent in $G$, and thus we have $|J|\le s-|B'|=s^*$, as desired. 
		
		Regarding the set 
		\[
			Z^* = \{v \in V(G^*)\sm A^* \colon \deg_{G^*}(v)> s^*\}
		\]
		we contend  
		\begin{equation}\label{eq:Z*}
			Z^*\subseteq V(G)\setminus (A\cup B\cup C)\,.
		\end{equation}
		Indeed, if $v\in V(G^*)\cap A$, then 
		$\deg_{G^*}(v)= \deg_G(v)-|A'|\le s-(3s-n)=s^*$
		and the same reasoning applies to $B$ and $C$ in place of $A$ as well. 
		
		As a consequence of~\eqref{eq:Z*} we have 
		$Q^*\subseteq V(G^*)\setminus (A^*\cup Z^*)$,
		as required. The assumption $|B\cap C|\le n-2s$ implies 
		$|Q^*|=s - |B\cap C|\ge 3s-n = 3s^* - n^*$,
		so $Q^*$ is sufficiently large for our purposes. 
		
		Finally, if $Z'\subseteq Z^*$ is independent in $G^*$, then $Z'$ is also 
		independent in $G$ and~\eqref{eq:Z*} shows that $Z'$ has to be disjoint to $A$. 
		Thus, Fact~\ref{f:matchable} tells us that in $G$ there is a matching $M$ from 
		$Z'$ to $A$. By Claim~\ref{clm:37} and~\eqref{eq:Z*} such a matching 
		can only use the part $A\setminus (B'\cup C')$ of $A$, meaning that 
		$Z'$ is indeed matchable into $V(G^*)\setminus (A^*\cup Z'\cup Q^*)$.
		Thereby Claim~\ref{clm:1913} is proved.
	\end{proof}
	
	Now the foregoing claim and Proposition~\ref{f:k2strong} result in 
	\[
		e(G^*)\le {n^*}^2-4n^*s^* + 5{s^*}^2 = 5n^2-24ns+29s^2\,,
	\]
	so
	\begin{align*}
		e(G) & = |B||B'|+|C||C'|+|A\setminus (B'\cup C')||A'|+e(G^*) \\
			&\le 2s(3s-n)+(2n-5s)(3s-n)+(5n^2-24ns+29s^2) \\
			&=3n^2-15ns+20s^2\,. \qedhere
	\end{align*}
\end{proof}
	
Perhaps somewhat surprisingly, the previous result allows us to study a similar 
configuration, where $C$ is no longer disjoint to one of $A$ and $B$. 
	
\begin{lemma}\label{f:ABC}
	Suppose $n\ge s\ge 0$ and let $G$ be a triangle-free graph on $n$ vertices 
	with $\alpha (G) = s$ containing two disjoint independent sets $A$ and $B$ 
	of size $s$, which in turn have subsets $A'\subseteq B$, $B'\subseteq A$ 
	with $|A'|=|B'|=3s-n$ and 
	\[
		K(A',A)\cup K(B',B)\subseteq E(G)\,.
	\]
	If $G$ contains a further independent set $C$ of size $s$ intersecting both $A$ 
	and $B$, then
	\[
		e(G) \le 3n^2 - 15ns + 20s^2\,.
	\]
\end{lemma}
\begin{proof}
	Regarding $n$ and $s$ as being fixed we consider a counterexample with $|Q|$ maximal, 
	where
	\[
		Q = \{q\in V\colon \Nn(q)=C\}\,.
	\]

	Let us start with a few basic observations. First of all, we have $C\cap A'=\emptyset$,
	because $K(A',A)\subseteq E(G)$ and $A$ intersects $C$. Similarly, one checks
	$C\cap B'=\emptyset$. Furthermore, since $\Delta(G)\le \alpha(G)=s$, the sets $(C\setminus A)\cup A'$ 
	and $(C\setminus B)\cup B'$ are independent. Hence they consist of at most $s=|C|$ 	
	vertices each, which yields  
	\begin{equation}\label{eq:CcapA}
		|A\cap C| \ge |A'| = 3s -n \qand |B\cap C| \ge |B'| = 3s -n\,.
	\end{equation}
	Finally, using the definition of $Q$, the independence of $A$, $B$, and $C$, and 
	the fact that $C$ meets $A$ and $B$, one obtains
	\begin{equation}\label{eq:QABC}
		Q \cap (A\cup B \cup C) = \emptyset \,.
	\end{equation}
	The maximality of $Q$  together with Lemma~\ref{f:3sets} leads to the following statement.
	
	\begin{claim}\label{cl:1}
		We have $V(G)=A \cup B \cup C \cup Q$.
	\end{claim}
		
	\begin{proof}
		Assume contrariwise, that there exists a vertex 
		$q\in V(G)\setminus (A\cup B\cup C\cup Q)$ and set $G'=\GZS(G\,|\, C, q)$. The sets 
		$A$, $B$, $A'$, $B'$, and $C$ still satisfy the hypothesis of Lemma~\ref{f:ABC}
		in $G'$. Moreover, Fact \ref{f:NRO} tells us that $G'$ is a triangle-free graph 
		with $e(G')\ge e(G)$. Due to the maximality of $|Q|$ all this is only 
		possible if $\alpha (G')>s$. This means that there exists an independent set 
		$D'=\{q\}\cup D\subseteq V(G')$ in $G'$ with $|D'|>s$. As usual, $D$ needs to be 
		an independent set in $G$ with $|D|=s$ and $C \cap D=\emptyset$.
		
		The case $k=3$ of Fact~\ref{f:kn<s<(k-1)n} allows us to assume $s>3n/8$,
		which leads to 
		\[
			|A'| + |B'| + |C| + |D| = 8s - 2n > n\,.
		\]
		Thus $D\cap (A'\cup B')\neq\emptyset$ and without loss of generality 
		we may suppose $D\cap A'\neq\emptyset$, which in turn implies $D\cap A = \emptyset$. 
		Moreover, we have 
		\[
			|B\cap D|\le |B\setminus C| = |B| - |B\cap C| 
			\overset{\eqref{eq:CcapA}}{\le} 
			n-2s\,.
		\]

		So altogether the graph $G$ and the sets of vertices $A$,~$B$, and~$D$ satisfy 
		the assumption of Lemma~\ref{f:3sets} with $D$ here in place of $C$ there. 
		But now $e(G) \le 3n^2 - 15ns +20 s^2$ contradicts~$G$ being a counterexample
		and, hence, establishes Claim~\ref{cl:1}.
	\end{proof}
		
	Continuing the proof of Lemma~\ref{f:ABC} we observe that, owing to the definition 
	of $Q$, the set $(A\setminus C)\cup Q$ is independent, and so $|Q|\le |A\cap C|$. 
	In combination with Claim~\ref{cl:1} this yields
	\[
		n 
		\overset{\eqref{eq:QABC}}{=} 
		|A \cup B \cup C|+|Q| \le 3s - |A \cap C| - |B \cap C| + |A\cap C|
		=3s - |B\cap C|\overset{\eqref{eq:CcapA}}{\le} n\,,
	\]
	and equality holds throughout. In particular, we have $|Q|= |A\cap C|$ and 
	consequently $E=(A\setminus C)\cup Q$ is an independent set consisting of $s$ vertices.
	To complete the proof it suffices to show that the assumptions of Lemma~\ref{f:3sets} 
	are satisfied for $B$, $A$, and $E$ here in place of $A$, $B$, and $C$ there. 
	The condition $B\cap E = \emptyset$ is clear and in the light of~\eqref{eq:CcapA}
	we obtain 
	\[
		|A\cap E| = |A\setminus C| = |A| - |A\cap C|
		\overset{\eqref{eq:CcapA}}{\le} 
		s-(3s-n)=n-2s\,. \qedhere
	\]
	\end{proof}
	
We proceed to the main result of this article, which we reformulate as follows. 

\begin{thm}\label{l:k3}
	If $n\ge s\ge 0$, then 
	\begin{equation}\label{eq:1929}
		\ex\le g_3(n,s)=3n^2-15ns+20s^2\,.
	\end{equation}
	Moreover, for $s\in\bigl[\frac 38 n, \frac 25 n\bigr]$ equality holds. 
\end{thm}
\begin{proof}
	The statement on equality follows from the first part in view of Fact~\ref{fact:1744}
	(see also Figure~\ref{fig:A8}), so it remains to establish~\eqref{eq:1929}.
	 
	Arguing indirectly we take a counterexample $(n, s)$ with $n$ minimum. 
	The case $k=3$ of Fact~\ref{f:kn<s<(k-1)n} tells us
	\begin{equation}\label{eq:sn3}
		\frac n3 < \frac{3n}8 < s < \frac{2n}5 < \frac n2\,.
	\end{equation}
	By Lemma~\ref{f:twosets} there exists a graph $G\in\Ex$ which contains 
	two disjoint independent sets~$A$ and $B$ of size $s$ such that there are 
	sets $A'\subseteq B$, $B'\subseteq A$ with $|A'|=|B'|=3s-n\overset{\eqref{eq:sn3}}{>}0$ 
	and 
	\[
		K(A',A)\cup K(B',B)\subseteq E(G)\,. 
	\]
	Pick two arbitrary vertices $a\in A'$, $b\in B'$ and set $G'=G-\{a,b\}$. 
	We shall consider two cases depending on the independence number of $G'$.
	
	\smallskip

	{\it \hskip1em  Case 1.\quad  $\alpha (G')=s$.} 

	\medskip

	Take an independent set $C\subseteq V(G')$ of size $s$. Since the set $C\cup \{a\}$ 
	has size $s+1$ and thus fails to be independent in $G$, we have 
	$C\cap A\neq \emptyset$. A similar argument shows $C\cap B\neq\emptyset$ and, hence,
	$C$ is as demanded by Lemma~\ref{f:ABC}. Consequently we have 
	$e(G) \le 3n^2 -15 ns + 20s^2$, which contradicts $(n, s)$ being counterexample.

	\smallskip

	{\it \hskip1em  Case 2.\quad $\alpha (G') \le s-1$.} 

	\medskip

	Observe that the minimality of $n$ leads to 
	\begin{align*}
		 e(G') & \le  3(n-2)^2 - 15(n-2)(s-1)+20(s-1)^2 \\
		   & =  (3n^2 - 15ns + 20s^2) + (3n - 10s + 2)\,.
	\end{align*}
	Since~\eqref{eq:sn3} yields $3n +1\le 8s$, this implies
	\[ 
		e(G') \le (3n^2 - 15ns + 20s^2) + (1-2s)
	\]
	and consequently we have 
	\[
		e(G) = e(G') +(2s-1) \le 3n^2 - 15ns +20s^2\,,
	\]
	which again contradicts the assumption that $(n, s)$ is a counterexample. 
\end{proof}


\end{document}